\documentclass{amsart}
\usepackage{amsmath,amsthm,amssymb,amscd}
\usepackage{graphicx}
\usepackage{tikz-cd}
\usepackage{paralist}
\usepackage{environ}
\usepackage{xcolor}
\usepackage{hyperref}
\usepackage{centernot}
\usepackage{mathtools}
\usepackage{mathrsfs}
\usepackage{marvosym}
\usepackage{xcolor,cancel}

\usepackage{booktabs,multirow}

\usepackage{pstricks}

\makeatletter
\def\square{\pst@object{square}}% reads star and options and continues with \square@i
\def\square@i(#1,#2)#3{{\use@par\solid@star\psframe[origin={#1,#2}](#3,#3)}}
\makeatother

\makeatletter
\DeclareFontFamily{U}{tipa}{}
\DeclareFontShape{U}{tipa}{bx}{n}{<->tipabx10}{}
\newcommand{\arc@char}{{\usefont{U}{tipa}{bx}{n}\symbol{62}}}%

\newcommand{\arc}[1]{\mathpalette\arc@arc{#1}}

\newcommand{\arc@arc}[2]{%
  \sbox0{$\m@th#1#2$}%
  \vbox{
    \hbox{\resizebox{\wd0}{\height}{\arc@char}}
    \nointerlineskip
    \box0
  }%
}
\makeatother

\makeatletter
\newcommand{\doublewedge}{\big@doubleop{\wedge}}
\newcommand{\big@doubleop}[1]{%
  \DOTSB\mathop{\mathpalette\big@doubleop@aux{#1}}\slimits@
}

\newcommand\big@doubleop@aux[2]{%
  \sbox\z@{$\m@th#1#2$}%
  \makebox[1.35\wd\z@][s]{$\m@th#1#2\hss#2$}%
}

\newcommand{\abs}[1]{\left|#1\right|}     %usage: \abs{x} yields |x|.
\newcommand{\near}{\delta} % Cech proximity
\newcommand{\dHtop}{\tau_H^{\Phi}} % descriptive proximal topology
\newcommand{\dnear}{\delta_{\Phi}} % descriptive proximity
\newcommand{\donear}{\delta_{\Phi_o}} % relaxed descriptive proximity

\usepackage{pstricks,pst-text,pst-grad,pst-node,pst-3dplot,pstricks-add,pst-poly,pst-coil,pst-all,pst-plot,pst-2dplot} % pspicture is in pstricks
\usepackage{pst-fun,pst-blur}
\usepackage{pst-all}
\usepackage{subfigure}
\renewcommand{\thesubfigure}{\thefigure.\arabic{subfigure}}
\makeatletter
\renewcommand{\p@subfigure}{}
\renewcommand{\@thesubfigure}{\thesubfigure:\hskip\subfiglabelskip}
\makeatother

\theoremstyle{plain}
\newtheorem{theorem}{Theorem}
\newtheorem{lemma}{Lemma}

\newtheorem{remark}{Remark}
\newtheorem{definition}{Definition}

\newtheorem{example}{Example}
\newtheorem{corollary}{Corollary}
\newtheorem{axiom}{Axiom}

\usepackage[title]{appendix}

\makeatletter
\@namedef{subjclassname@2020}{\textup{2020} Mathematics Subject Classification}
\makeatother

\begin{document}

\title{Indefinite Descriptive Proximities Inherent in Dynamical Systems. An Axiomatic Approach}

\author[J.F. Peters]{James Francis Peters}
\address{
Computational Intelligence Laboratory,
University of Manitoba, WPG, MB, R3T 5V6, Canada and
Department of Mathematics, Faculty of Arts and Sciences, Ad\.{i}yaman University, 02040 Ad\.{i}yaman, Turkiye, \url{https://orcid.org/https://0000-0002-1026-4638}.
}
\email{james.peters3@umanitoba.ca}
\thanks{The research has been supported by the Natural Sciences \&
Engineering Research Council of Canada (NSERC) discovery grant 185986 
and Instituto Nazionale di Alta Matematica (INdAM) Francesco Severi, Gruppo Nazionale per le Strutture Algebriche, Geometriche e Loro Applicazioni grant 9 920160 000362, n.prot U 2016/000036 and Scientific and Technological Research Council of Turkey (T\"{U}B\.{I}TAK) Scientific Human
Resources Development (BIDEB) under grant no: 2221-1059B211301223.}

\author[T. Vergili]{Tane Vergili}
\address{
Department of Mathematics, Karadeniz Technical University, Trabzon, Turkiye,
\url{https://orcid.org/0000-0003-1821-6697}.
}
\email{tane.vergili@ktu.edu.tr}

\author[F. Ucan]{Fatih Ucan}
\address{
	Department of Mathematics, Karadeniz Technical University, Trabzon, Turkiye, \url{https://orcid.org/0000-0002-6975-9408}.
}
\email{fatihucan@ktu.edu.tr}

\author[D. Vakeesan]{Divagar Vakeesan}
\address{
Computational Intelligence Laboratory,
University of Manitoba, WPG, MB, R3T 5V6, Canada, \url{https://orcid.org/0000-0001-6708-4585}.
}
\email{divagarv@myumanitoba.ca}


\date{}

\begin{abstract}
This paper introduces indefinite proximities inherent in the collection of physical objects found in a dynamical system.  Axiomatically, these indefinite proximities lead to a new form of Hausdorff topology, which is indefinite descriptively.  The main results in this paper are (1) Every descriptive proximity space on a dynamical system is indefinite (Theorem 1), (2) 
Every dynamical system has an indefinite descriptive Hausdorff topology (Theorem 3), and (3) The energy of a dynamical system varies with every clock tick (Theorem 4).
An application of these results is given in terms of the detection of those portions of a dynamical system that are stable and that have low energy dissipation.
\end{abstract}
%
%\keywords[2020 Mathematics Subject Classification]{054E05 (Proximity); 54C50 (Topology)}
\maketitle
\tableofcontents

\section{Introduction}
This paper introduces an axiomatic approach in the study of measurable descriptive proximities that are inherent in the self-similarities in the parts, behaviours and waveforms of complex dynamical system.  The consequences of this approach carry forward recent work on the descriptive approach in the study of dynamical systems~\cite{Sharkovsky2022,PetersLiyanage2024,ErdagPetersDeveci2024Dnear0}, especially chaotic dynamical systems~\cite{Edelman2018}. This approach leads to the introduction of a number of new descriptive proximity relations that represent advances in proximity space theory~\cite{Csaszar1987catProx,Naimpally2013,PetersVergili2021DescriptiveHomotopy} useful in the detection and measurement of those portions of dynamical systems that are stable and that have low energy dissipation.

All proximities between the parts and waveforms of dynamical systems are indefinite.  This observation leads to the introduction of an indefinite descriptive proximity $\delta_{lim\Phi}$, which is a refinement of the relaxed descriptive proximity $\donear$~\cite{ErdagPetersDeveci2024Dnear0}. We observe that every descriptive proximity space on a dynamical system is indefinte (Theorem~\ref{theorem:indefiniteDynamicalSys}). Important results stemming from the introduction of an indefinite descriptive distance $d^{lim\Phi}$ (Definition~\ref{def:dlim}) are given, namely, Indefinite Descriptive Hausdoff Topology (Lemma~\ref{lemma:indefiniteTop}) and every dynamical system has an indefinite Hausdorff topology (Theorem~\ref{theorem:indefiniteDescrHausdorffTop}). This paper also  includes an application of $\delta_{lim\Phi}$ in detecting as well as measuring the stability, low energy dissipation portions of dynamical system waveforms.

\begin{table}[h!]\label{table:symbols}
\caption{Principal Symbols Used in this Paper}
\begin{tabular}{|c|c|}
%\toprule
\hline
Symbol & Meaning\\ 
\hline
$d^{\Phi}$ & Descriptive proximity distance:~Def.~\ref{def:dnear}\\
\hline
$d^{\Phi}_{H}$ & Descriptive Hausdorff distance:~Def.~\ref{def:HausdorffDistance2}\\
\hline
$(\mathcal{K}_{\Phi},d_H^\Phi,\dHtop)$ & Descriptive Hausdorff Topological Space:~Def.~\ref{def:topology2}\\
\hline
$d^{lim\Phi}$ & Indefinite descriptive distance:~Def.~\ref{def:dlim}\\
\hline
$\left(X,f,\Phi\right)$ & Descriptive dynamical system:~Def.~\ref{def:DDS}\\
\hline
$Per(f, \Phi)$ & Descriptive  periodic points:~Def.~\ref{def:descriptivePeriodicPts}\\
\hline
$\donear$ & Relaxed descriptive proximity:~Def.~\ref{def:donear}\\
\hline
$E_{m(t)}$ & Waveform $m(t)$ energy:~Def.~\ref{axiom:waveformEnergy}\\
\hline
$E_{diss}(loc,t)$  & Energy dissipation at location $loc$ at time $t$:~Def.~\ref{def:Ediss}\\
\hline
\end{tabular}
\end{table}
\section{Preliminaries}
This section introduces descriptive proximities as well as descriptively proximal self-similarities in dynamical system behaviors.

Let $X$ be a nonempty set, $2^X$ denote the collection of subsets of $X$, and $A\in 2^X$ for a nonempty set $A$ with $n$ characteristics. 
A \textit{probe function} on $X$ is a mapping $\phi: 2^X \to \mathbb{R}$ and  a  characteristic of a subset $A$ is  $\phi(A) \subset \mathbb{R}$. In that case, a \textit{complete description} of $A$ with $n$ characteristics is a set $\Phi(A)$ where \boxed{\Phi: 2^X \to \mathbb{R}^n,n\geq 1} is a mapping defined by 
\begin{center}
\boxed{\Phi(A) =  \left\{ \left(\phi_1(a),\phi_2(a),\dots,\phi_n(a)\right)  \ \ : \  a\in A \right\} \subseteq \mathbb{R}^n.}
\end{center}

Notice that $\Phi(A)\cap\Phi(B)\neq \emptyset$ implies  there exist $a\in A$ and $b \in B$ such that $\Phi(\{a\})=\Phi(\{b\})$. Throughout the rest of the paper,  we will simply use $\Phi(a)$ instead of $\Phi(\{a\})$.
	
\begin{definition}\label{def:dnear}{\bf (Descriptively Proximal Sets~\cite{Naimpally2013}).}  
Let $X$ be a nonempty set and $A,B\in 2^X$.  
Consider the descriptive proximity mapping $d^{\Phi}: 2^X\times 2^X\to\mathbb{R}$ defined by
	\begin{center}
		\boxed{d^{\Phi}(A,B)=\inf_{\substack{a\in A \\ b\in B}} \abs{\Phi(a) - \Phi(b)} = r\in\mathbb{R}^{\geq 0}.
	}
	\end{center}
Then we say that $A,B$ are descriptively proximal, which is denoted by $A\dnear B$ provided 
 $d^{\Phi}(A,B) = 0$. 
	\qquad \textcolor{blue}{$\blacksquare$}
\end{definition}

\begin{figure}[!ht]
	\centering
	\includegraphics[width=80mm]{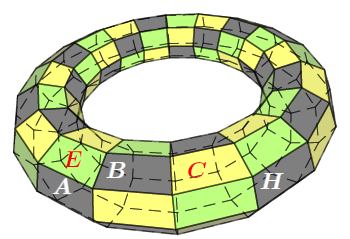}
	\caption{Torus near sets}
	\label{fig:torus}
\end{figure} 

\begin{example}
	Let \boxed{\phi(A) = k\ \mbox{nm (nanometers)}\in \mathbb{R}} be the description of a torus panel shown in Figure~\ref{fig:torus}, limited to a single characteristic, namely, a panel color wavelength in the visible spectrum.  Assume each gray panel has a wavelength = 304 nm.  Consider the rectangular panels $A,B,C,E,H$ on the surface of the torus in Figure~\ref{fig:torus}.
	Notice that 
	\begin{center}
		\boxed{\phi(A) = \phi(B) = \phi(H) = 304\ nm.}
	\end{center}
	In addition, we have the following descriptive proximities:
	\begin{compactenum}[1$^o$]
		\item \boxed{A\not{\dnear} E}   (panel $A$ is not descriptively near\ panel $E$)
		\item \boxed{A\not{\dnear} C}  (panel $A$  is not descriptively near  panel $C$)
		\item \boxed{A \dnear B}  (panel $A$ is  descriptively near  panel $B$)
		\item \boxed{A \dnear H} (panel $A$ is  descriptively near  panel $H$)
		\item \boxed{B \dnear H}  (panel $B$ is   descriptively near  panel $H$)
		\item \boxed{C \not{\dnear} H}  (panel $C$  is not descriptively near  panel $H$)
		\item \boxed{H \not{\dnear} E}  (panel $H$ is not descriptively near  panel $E$) 
		\item \boxed{A \dnear A}  (self-descriptive-proximity) \qquad \textcolor{blue}{$\blacksquare$}
	\end{compactenum}
\end{example}

\begin{remark}
Descriptions of dynamically changing systems such as observable ring water in each of the rings on the planet Saturn~\cite{SaturnRingRain2015} are usually incomplete, since the number of known characteristics is typically incomplete.  For example, in describing the proximities between self-similarities in a physical chaotic system represented by the collection of subsets $2^X$ for a set $X$ of system objects, the description of a subset $A\in 2^X$ would typically be incomplete.
This is consistent with the view that the characteristics of every physical object is indefinite (see Axiom~\ref{axiom:Physical}) and the descriptions of every pair of physical objects is not fixed (see Axiom~\ref{axiom:PhysicalPairs}).
\qquad \textcolor{blue}{$\blacksquare$}
\end{remark}

\begin{axiom}\label{axiom:Physical}
Let $\Phi(A)$ be a complete description of a nonempty set of physical objects $A$ and let $\abs{\Phi(A)}$ be the size of $\Phi(A)$.  The number of characteristics of $A$ is indefinite, i.e., 
\boxed{\abs{\Phi(A)} = k\in \mathbb{Z^+}}.
\end{axiom}

\begin{axiom}\label{axiom:PhysicalPairs}
For each pair of sets of physical objects $A$ and $B$, the difference between the descriptions is not fixed, i.e.,
%{\color{blue} 
\boxed{d^{\Phi}(A,B)= r\in \mathbb{R}^{\geq 0}, i.e.,\abs{\Phi(A)}-\abs{\Phi(B)} = r}. 
%}  
%{\color{red} Dear Professor James Peters, Is that exactly what you mean, or do you mean $\abs{\Phi(A)}-\abs{\Phi(B)}=r$?
%}
\end{axiom}

\begin{definition}\label{def:dnear2}{\bf (Descriptive Proximity Space ~\cite{Naimpally2013})}.\\
	For a nonempty set $X$, 
	\boxed{\left(2^X,\dnear\right)} is a descriptive proximity space. 
	\quad \textcolor{blue}{$\blacksquare$}
\end{definition}

\begin{definition}\label{def:HausdorffDistance}{\bf (Hausdorff Distance~\cite{Hausdorff1914,Hausdorff1914b})}.
The Hausdorf distance, $d_H(Q,S)$, between a pair of compact subsets $Q$ and $S$ in  $\mathbb{R}^m$ is defined by
\begin{center}
	\boxed{d_H(Q,S) = \max \left\{ \sup_{\substack{q\in Q}} D(q,S), \sup_{\substack{s\in S}} D(Q,s)\right\} \geq 0}
\end{center}
where $D(p,-)$ or $D(-,p)$ denote the distance between a single point $p$ and a given set. \textbf\qquad \textcolor{blue}{$\blacksquare$}
\end{definition}

%{\color{blue}
One can also measure the descriptive Hausdorff distance of $A$ and $B$ in $X$ by assuming that their complete descriptions $\Phi(A)$ and $\Phi(B)$ are compact subsets of $\mathbb{R}^n$.

\begin{definition}\label{def:HausdorffDistance2}{\bf (Descriptive Hausdorff Distance)}\\
Let $(X,\dnear)$ be a descriptive proximity space with a collection of $n$ characteristics, and $A,B\in 2^X$ with compact complete descriptions  $\Phi(A),\Phi(B)$ in $\mathbb{R}^m$.  The \textit{descriptive Hausdorff distance} $d^{\Phi}_H(A,B)$ between $A$ and $B$ is defined by
\begin{center}
\boxed{d^{\Phi}_H(A,B) = d_H(\Phi(A), \Phi(B))}
\end{center}
\textbf\qquad \textcolor{blue}{$\blacksquare$}
\end{definition}
%}

\begin{remark}
	We assume that
	\begin{center} 
		\boxed{d^{\Phi} (A,B) = d_H(\Phi(A),\Phi(B)) = r\in \mathbb{R}^{0+}}. 
	\end{center}
	for a set of known characteristics for $A,B\in 2^X$ for a set $X$.  In a dynamical system $X$ that is chaotic, $X$ is inherently self-symmetic.  
	\textbf\qquad \textcolor{blue}{$\blacksquare$}
\end{remark}

%{\color{blue}
Given a descriptive proximity space $(X,\dnear)$, let $\mathcal{K}_{\Phi}$ be a collection of all subsets of $X$ with compact complete descriptions
\[\mathcal{K}_{\Phi}=\{A \in 2^X \ : \   \Phi(A) \ \mbox{is compact} \}.\]
Then $d^{\Phi}_H$ defines a metric on $\mathcal{K}_{\Phi}$.
%}

%{\color{blue}
\begin{definition}\label{def:topology2}{\bf (Descriptive Hausdorff Topological Space)}.\\
A descriptive Hausdorff topology, $\dHtop$, induced by the descriptive Hausdorff metric $d_H^{\Phi}$  on $\mathcal{K}_{\Phi}$ has the following properties.
	\begin{compactenum}[1$^o$]
		\item $\mathcal{K}_{\Phi},\emptyset \in \dHtop$.
		\item $\{\mathcal{A}_i\}_{i\in I}  \subseteq \dHtop$ implies  $\bigcup_{i\in I} \mathcal{A}_i \in \dHtop$.
		\item $\mathcal{A}, \mathcal{B} \in \dHtop$ implies $\mathcal{A} \cap \mathcal{B} \in \dHtop$.
	\end{compactenum}
The triple \boxed{(\mathcal{K}_{\Phi},d_H^\Phi,\dHtop)} is called a descriptive Hausdorff topological space. 
	\qquad \textcolor{blue}{$\blacksquare$}
\end{definition}  
%}

%{\color{red} Dear Professor James Peters, I have changed the theorem below into a corollary. Is that ok with you?
%}
\begin{corollary}\label{theorem:fundamental}
	There is a descriptive Hausdorff topology $\dHtop$ on every collection of compact complete descriptions of a  descriptive proximity space $(X,\dnear)$.
\end{corollary}
%\begin{proof}
%	Replace  $2^{\Phi(X)}$ (collection of descriptions) in Def.~\ref{def:topology2} with $2^{\dnear(X\times X)}$ (collection of near sets) and desired result follows.
%\end{proof} 

\begin{remark}
From Axiom~\ref{axiom:Physical}, the number of characteristics of a physical objects is indefinite.  For this reason, we introduce a new form of descriptive proximity of a pair of sets of physical objects $A,B$ in terms of the difference of the descriptions of $A$ and $B$ converging to 0 in the limit (see Definition~\ref{def:dlim}).
\qquad \textcolor{blue}{$\blacksquare$}
\end{remark}

\begin{definition}\label{def:dlim}{\bf (Indefinite Descriptive Distance)}.\\
Let $X$ be a nonempty set.  $A,B\in 2^X$ are indefinitely descriptively near sets of physical objects, 
%{\color{red} denoted by $A \near_{\lim\Phi} B$} 
provided 
\begin{center}
\boxed{d^{\lim \Phi}(A,B)=\lim \limits_{\abs{\Phi(A)-\Phi(B)}\to 0} d^{\Phi}(A,B) = 0}\qquad \textcolor{blue}{$\blacksquare$}
\end{center}
\end{definition}

\begin{example}
Assume that the torus in Figure~\ref{fig:torus} is one of the rings of Saturn and that the rectangular plates in the torus Figure~\ref{fig:torus} are sets of particles (e.g., ring containing charged water molecules~\cite{SaturnRingRain2015}) $A,B,C,E,H$ swept along within a Saturn ring. Then observe
\begin{compactenum}[1$^o$]
\item Each set of ring water molecules has an indefinite number of characteristics (from Axiom~\ref{axiom:Physical}).
\item Each pair of sets of ring water molecules are descriptively indefinite near sets (from Axiom~\ref{axiom:PhysicalPairs}).  This observation makes it possible to organize the descriptions of ring water sets in a concise way.
%\item Let $X$ be a set of ring water particles and let $2^X$ be a collection of subsets of ring particles.  Let a descriptive Hausdorff proximity space $(2^X\times 2^X,\dnear)$ be equipped with the Hausdorff metric $d_H$.
%From Lemma~\ref{lemma:locallyCompact}, the descriptive Hausdorff topological space $\dtop$ on this proximity space is locally compact.  This observation paves the way for the introduction of a variety of different open covers on $\dtop$, which would be a source of useful descriptions of Saturn ring water regions. 
\end{compactenum}
\end{example}

\begin{figure}[!ht]
	\centering
	\includegraphics[width=100mm]{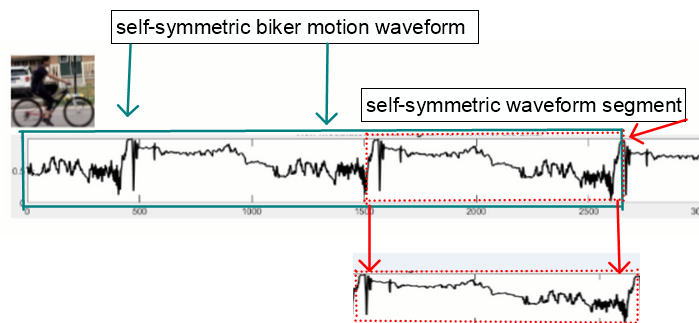}
	\caption{Self-similar biker motion waveform}
	\label{fig:biker}
\end{figure} 

%\begin{figure}[!ht]
	%\centering
	%\includegraphics[width=80mm]{self-similar-waveform}
	%\caption{Sample self-similar waveform}
	%\label{fig:waveform}
%\end{figure}

%Guozhen2005	
%SaturnRingRain2015

\section{Chaotic dynamical system}
Dynamical chaos is different from randomness or disorder.  Chaos is deterministic in descriptive set theory~\cite{Sharkovsky2022}.   For example, vibrations of a mechanical system are quasi-periodic, changing from period doubling to chaos with period doubling~\cite{WeiChaos2017}.  Chaos in dynamics occurs if the cloud of representative points in the course of evolution in the phase space undergoes repetitive deformations of stretching, folding, and transversal compression~\cite{KuznetsovHyperbolicChaos2012}.   Typically, nonlinear motion cascades to chaos~\cite{ZhaoNonlinearDynamics2024}.

\begin{definition}\label{def:system}{\bf(System)}.\\
A system is a configuration of components functioning together as a whole and [in] their relationships.
\qquad \textcolor{blue}{$\blacksquare$}
\end{definition}

\begin{definition}\label{def:dynamicalSystem}{\bf(Dynamical System)}.\\
A dynamical system is a physical system (collection of physical objects) that changes over time. 
\qquad \textcolor{blue}{$\blacksquare$}
\end{definition}

\begin{definition}\label{def:chaoticSyStem}{\bf(Chaotic Dynamical System)}.\\
A chaotic dynamical system (CDS) is a system that has the following properties:
\begin{compactenum}[1$^o$]
\item The waveform of a CDS is self-similar, i.e., the CDS contains parts with the same structure as the complete CDS.
\item The self-similar fractals in the CDS satisfy the open set condition~\cite{Edelman2018}, i.e., there exists an integer $n$ so that for every piece $A_i\in CDS$ with diameter $\varepsilon$, there are at most incomparable pieces $A_{j\leq n}\in CDS$ with diameter $\geq \varepsilon$ with distance < $\varepsilon$ from $A_i$~\cite{Schief1994}.  Let $X\in 2^Y$ in a CDS with subsets $2^X$ in a descriptive Hausdorff proximity space \boxed{(\mathcal{K}^Y_{\Phi}, d_H^{\Phi}, \dHtop)}. 
A description $\Phi(X)$ of a CDS $X$ in a space $Y$ contains the description of all the subsets of $A\in 2^X$ such that $d_H(A,B)<\varepsilon$ for all $B\in 2^Y$.   
\item The number of known charactertics of the CDS is indefinite.
\end{compactenum}
\qquad \textcolor{blue}{$\blacksquare$}
\end{definition}

\begin{example}
A vibrating syswtem is an example of CDS.  For example, in Figure~\ref{fig:biker} a biker is a CDS with the following properties:
\begin{compactenum}[1$^o$]
\item The biker waveform in Figure~\ref{fig:biker} is self-similar, i.e., the complete motion waveform pattern is repeated in its segments.
\item The biker waveform satisfies the open set condition.  To see this, 
let $X\in 2^Y$ in a biker motion system with subsets $2^X$ in a collection $Y$ of moving systems (e.g., walkers and vehicles).  This biker CDS resides in a descriptive Hausdorff proximity space $\mathcal{K}^X_{\Phi}$ that is a subspace of $\mathcal{K}^Y_{\Phi}$.  Let $B\in 2^Y$.
A description $\Phi(A)$ of a subset $A\in 2^X$ in space $Y$ contains the description of all the subsets of $A\in 2^X$ such that $d_H(A,B)<\varepsilon$ for all $B\in 2^Y$.
\item The number of known characteristics of the biker system is indefinite.
\end{compactenum}
\qquad \textcolor{blue}{$\blacksquare$}
\end{example}

\begin{lemma}\label{lemma:indefiniteDescriptively}
Every descriptive proximity space on a collection of physical objects is indefinite.
\end{lemma}
\begin{proof}
Let $(X,\dnear)$ be a descriptive proximity space on a collection of physical objects in $2^X$.
From Axiom~\ref{axiom:Physical}, the number of known characteristics of subsets $A,B\in 2^X$ is indefinite.  Then, from Definition~\ref{def:dlim}, 
 $d^{\Phi}(A,B) = d^{\lim\Phi}(A,B)$.  Hence, $(2^X,d^{\Phi}) = (2^X,d^{\lim\Phi})$ is indefinite.
\end{proof}

\begin{theorem}\label{theorem:indefiniteDynamicalSys} Every descriptive proximity space on a dynamical system is indefinite.
\end{theorem}
\begin{proof}
Immediate from Axiom~\ref{axiom:Physical} and Lemma~\ref{lemma:indefiniteDescriptively}.
\end{proof}

\begin{theorem}\label{theorem:indefiniteDescriptively}
Every descriptive proximity space on a chaotic dynamical system is indefinite descriptively.
\end{theorem}
\begin{proof}
Immediate from Definition~\ref{def:chaoticSyStem} and Theorem~\ref{theorem:indefiniteDynamicalSys}.
\end{proof}

\begin{lemma}\label{lemma:indefiniteTop}{\bf(Indefinite descriptive Hausdorff Topology)}.\\
Every descriptive Hausdorff topology on a sets of objects in a physical system is indefinite.
\end{lemma}
%{\color{red} Dear Professor James Peters, I have changed some parts that are up in the text, and now with these changes  I can't understand the proof down below. Would you please reconsider editing this proof if everything up here seems to be correct to you?}

\begin{proof}
Let $2^X$ be collections of subsets of a dynamical system $X$, $A,B\in 2^X$.  Let $(X,\near_{\lim\Phi})$ be an indefinite descriptive proximity space. From Axiom~\ref{axiom:PhysicalPairs}, all descriptive Hausdorff distances \boxed{d_H(\Phi(A),\Phi(B))} approach zero, i.e., descriptive Hausorff proximities between sets of physical objects in a dynamical system are indefinite.  By replacing $2^X$ in Definition~\ref{def:topology2}, and obtain an  indefinite descriptive Hausdorff topology.
\end{proof}

\begin{theorem}~\label{theorem:indefiniteDescrHausdorffTop}
Every Dynamical System has an indefinite descriptive Hausdorff topology.
\end{theorem}
\begin{proof}
Immediate from Definition~\ref{def:dynamicalSystem} and Lemma~\ref{lemma:indefiniteTop}.
\end{proof}

\section{Descriptive Case}
There are many definitions of chaos in the literature. Among these definitions, we often encounter definition of chaos in sense of R.L. Devaney made in 1986~\cite{Devaney1986}, followed by studies of the relationship between individual chaos and collective chaos (see, e.g.,~\cite{Schief1994,Edelman2018,HaiderPeters2021self-similar,OzkanKulogluPeters2021self-similarilty}).
In this section, our aim is to analyze a chaotic dynamical systems in the descriptive sense by adding new features to it and observe them.
  
\subsection{Descriptive Dynamical System}
The fabric of a dynamical system is a family of interactions on a nonempty set of points (quanta) $X$. There is an orbit of each point $x\in X$ defined by a family of interations on $x$. In this subsection, we consider $X$ as a topological space together with a probe function $\Phi: X \to \mathbb{R}^m$ on it.

%{\color{blue}
\begin{definition}\label{def:familyOfInterations} {\bf (Family of Interations).}	\\
  The dynamical system $(X,f)$  on $X$ defined by $f$ is the family of iterations $\{f^n\}_{n\in \mathbb{Z}^{+}}$ with each $f^n$ mapping from  $X$ to itself. For $x\in X$ and $n\in \mathbb{Z}^+$ the orbit of $x$ is the set of points 
  \[x,f(x),f^2(x),\dots, f^n(x),\dots\]
  and  is denoted by
\begin{center}
\boxed{
	Orb_f(x)=\{f^n(x) : n\in \mathbb{Z}^{+}\}.
	} \qquad \textcolor{blue}{$\blacksquare$}
\end{center}
\end{definition}
%}

%{\color{blue}
\begin{definition}\label{def:DDS}{\bf (Descriptive dynamical system).}\\	
One can also define the  orbit of the descriptions of a point $x$ in X  which is set of points
\[\Phi(x),\Phi(f(x)),\Phi(f^2(x)),\dots,\Phi(f^n(x)),\dots \]
Briefly, we call the dynamical system $(X,f)$ together with a probe function $\Phi$ a {\bf  descriptive dynamical system} which is denoted by $(X,f, \Phi)$.
\qquad \textcolor{blue}{$\blacksquare$}
\end{definition}

\begin{definition}\label{def:descriptiveFixedPoints} \cite{PetersVergili2021GoodCoverings}  Given a descriptive dynamical system $(X,f, \Phi)$, a subset $A$ of $X$ is  
 \begin{compactenum}
	\item  \textit{a descriptive fixed subset of $f$}, provided $\Phi(f(A)) = \Phi(A)$
	\item \textit{a descriptive period-m set of f} if $\Phi(f^m(A))=\Phi(A)$
	and \newline $\Phi(f^j(A))\neq \Phi(A)$ 
	for  the set of descriptive periodic sets $j = 1, 2, \ldots, m-1$. 
\end{compactenum}
\qquad \textcolor{blue}{$\blacksquare$}
\end{definition}

%{\color{blue} 
When $A=\{a\}$ is a singleton set,  we simply replace the word "set" with the word "point". In that case, we say that a point $a\in X$ is 
\begin{compactenum}
	\item  a \textit{descriptive fixed point of} $f$ if $\Phi(f(a))=\Phi(a)$. 
	\item a \textit{descriptive period-m point of $f$} if $\Phi(f^m(a))=\Phi(a)$ and $\Phi(f^j(a))\neq\Phi(a)$ for $j=1,2,\dotsc, m-1$.
\end{compactenum}
%}

\begin{definition}\label{def:descriptivePeriodicPts}{\bf (Descriptive Periodic Points).}\\ 
Given a descriptive dynamical system $(X,f, \Phi)$, 
\begin{compactenum}
	\item \textit{the set of descriptive m-periodic points} is given by 
	\[\mbox{Per}_m({f, \Phi})= \left\{a\in X : \Phi(f^m(a))=\Phi(a) \right\}\]
	\item \textit{the set of descriptive periodic points} is given by 
	\[\mbox{Per}(f, \Phi)= \bigcup_{m\in \mathbb{Z}^+} \mbox{Per}_m(f, \Phi)= \left \{a\in X \ : \  \Phi(f^m(a))=\Phi(a),m\in \mathbb{Z}^+ \right\}.\]
\end{compactenum}
\qquad \textcolor{blue}{$\blacksquare$}
\end{definition}

%{\color{blue}
\begin{remark}
	Here, one can understand from  Definition~\ref{def:chaoticSyStem} and  Theorem~\ref{theorem:indefiniteDescriptively} that every period-m set on a chaotic dynamical system is an indefinite descriptive period-m set. However, the converse may not be true. 
	Even so,  a period-m set is a descriptive period-k set for $k\leq m$. 
\end{remark}
%}

%{\color{red}  Now we consider a descriptive dynamical system $(X, f, \Phi)$  where $X$ is also a metric space. Let  $(X,f,\Phi)$  be a  descriptive dynamical system and  $\mathcal{K}(X)$ denote a collection of the nonempty compact subsets of $X$ together with a Hausdorff metric $d_H$ on it. 
%}

%{\color{blue}
\begin{definition}
	The set-valued function $\bar{f}: \mathcal{K}(X) \rightarrow \mathcal{K}(X)$ is defined by $\bar{f} (A)=\{f(a) : a\in A\}$ so that the complete description of $A$ is 
	\begin{align*}
\Phi(\bar{f}(A)) &=\left \{\Phi(f(a)) \ :  \ a\in A \right \}\\ &= \left\{\big(\phi_1(f(a)),\phi_2(f(a)),\dotsc,\phi_n(f(a))\big) : a\in K\right\}\subseteq \mathbb{R}^n
	\end{align*} 
	and its descriptive orbit is a collection of sets
	\[ \Phi({\bar{f}(A))},\Phi(\bar{f}^2(A)),\dotsc ,\Phi(\bar{f}^n(A)),\dotsc \]
\end{definition}

\begin{remark}
 	For a subset $A$ of $X$ the extension of $A$ to $\mathcal{K}(X)$ is defined by 	\begin{center}
 	$e(A)=\{K\in \mathcal{K}(X) \ :  \ K\subset A\}.$	
 \end{center}
\end{remark}
Also  the following lemma is given in \cite{Flores2003}.

\begin{lemma}
	 For two subsets A and B of X, the following holds.
	 \begin{compactitem}
	 	\item  $e(A)\neq \emptyset$ if and only if $A\neq \emptyset$
	 	 \item $e(A)$ is open subset of $\mathcal{K}(X)$, if $A\subseteq X$ is open.
	 	 \item $e(A\cap B)=e(A)\cap e(B)$
	 	 \item $\bar{f}(e(A))\subseteq e(\bar{f}(A))$
	 	 \item $\bar{f}^n = \bar{f^n}$, for all $m\in \mathbb{Z}^+$
	 \end{compactitem}
 \end{lemma}
Here we present a descriptive approach to Devaney's definition of chaos by introducing  new definitions: a dense collection of descriptive periodic sets, descriptive transitivity and descriptively sensitive.  
 \begin{definition}
 	A  mapping $\bar{f}:  \mathcal{K}(X) \rightarrow  \mathcal{K}(X)$ is said to be descriptively chaotic if  
 	\begin{compactitem}
 		\item the collection of $\mbox{Per}(\bar{f}, \Phi)$ descriptive periodic compact sets which is given by \[ \bigcup_{m\in \mathbb{Z}^+} \mbox{Per}_m(\bar{f}, \Phi)=\left \{K\in \mathcal{K}(X) \ : \  \Phi(\bar{f}^m(K))=\Phi(K),m\in \mathbb{Z}^+ \right\}\] is dense in $\mathcal{K}(X)$, 
 		\item for every $e(U),e(V)$ which are respectively extensions of  nonempty open pair $U,V \subseteq X$, there exists $n\in\mathbb{Z}^{+}$ such that  
 		$ \bar{f}(e(U))$ and $e(V)$ are descriptively proximal, i.e., $\Phi(\bar{f}^n(e(U)))\cap\Phi(e(V))\neq \emptyset$ ($\bar{f}$ is descriptively transitive), and 
 		\item there is a $\delta>0$ (sensitive constant) such that whenever $e(U)$ is an extension of a nonempty open subset $U$ of $X$, there exist $A,B$ in  $e(U)$ and  $n\in \mathbb{Z}^{+}$ such that $d^{\Phi}(\bar{f}^n(A),\bar{f}^n(B))=d_H(\Phi(\bar{f}^n(A),\bar{f}^n(B)))>\delta$
 	($\bar{f}$ is descriptively sensitive).
 	\end{compactitem}  
 \end{definition}
%}
 
%\begin{definition}
%	In a descriptive dynamical system, the function f is called a descriptively transitive in $X$, provided that	for any non-empty open pair of $A, B \subseteq X$ there exists $n\in\mathbb{Z}^{+}$ such that  $f^n(A)$ and $B$ are descriptively proximal, i.e., $\Phi(f^n(A))\cap \Phi(B)\neq \emptyset$ \big(there exist $a\in A$ and $b\in B$ such that $\Phi(f^n(a))=\Phi(b)\big)$. \qquad \textcolor{blue}{$\blacksquare$}
%\end{definition}

\begin{lemma}
If the collection of periodic sets in a Hausdorff metric space is dense, so is the collection of descriptively periodic sets.
\end{lemma}
\begin{proof}
Immediately follows from the fact that any periodic set is also descriptively periodic set. 
\end{proof}

\begin{lemma}
If $\bar{f}$ is topologically transitive, so is descriptively transitive. 
\end{lemma}
\begin{proof}	
Let $e(U),e(V)$  be extensions of two nonempty open subsets $U,V$ of $X$ respectively. Since $\bar{f}$ is topologically transitive, there exists  a positive integer $m$ such that $\bar{f}^m(e(U))\cap e(V) \neq \emptyset$. In that case there exists a compact subset  $K$ in $\mathcal{K}(X)$  such that $K\in \bar{f}^m(e(U))$ and $K\in e(V)$ . Since $K$ is in the image of $\bar{f}^m(e(U))$,  there is $K'$ in $e(U)$ so that $\bar{f}^m(K')=K$. This implies $\Phi(\bar{f}^m(K^{\prime}))= \Phi(K)$ i.e., $\Phi(\bar{f}^m(e(U)))\cap \Phi(e(V))\neq \emptyset$.
\end{proof}

%\begin{definition}
%	Let $(X,d)$ and $(\mathbb{R}^n,D)$ metric spaces,  $f:X\rightarrow X$ a map (continuos function)  and $\Phi:X \rightarrow \mathbb{R}^n$ a probe function. The $f$ is {\bf descriptively sensitive dependent on initial conditions} in $X$, provided that there exists $\delta>0$ such that for any $a\in X$ and for any neighborhood $N$ of $a$, there exist $b\in N$ and $n\in \mathbb{Z}^+$ such that $ D(\big(\Phi(f^n(a)),\Phi (f^n(b))\big)=r>\delta$. \qquad \textcolor{blue}{$\blacksquare$}
%\end{definition}

%\begin{definition}
 % For a descriptive dynamical metric space $(X,d, \Phi_f)$ a map $f: X \to X$ is called to be descriptively chaotic if 
  
  %\begin{compactitem}
%	\item The set of $Per(\Phi_f)$ is dense in $X$.
%	\item $f$ is descriptively transitive.
%	\item $f$ is descriptively sensitive on initial conditions.
 %  \end{compactitem}
%\textcolor{blue}{\Squaresteel}
%\end{definition}

\begin{figure}[!ht]
	\centering
	\includegraphics[width=125mm]{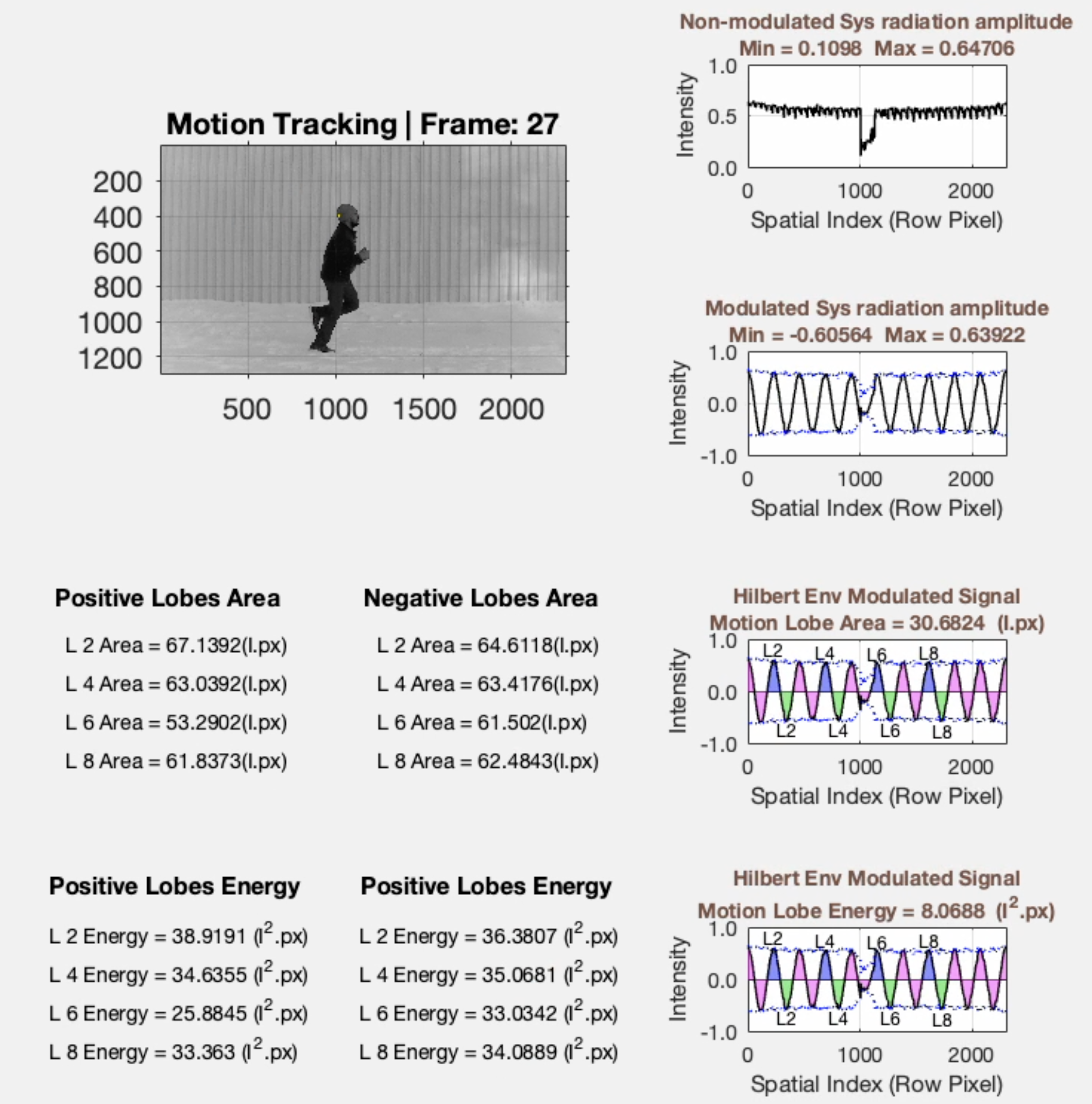}
	\caption{Relaxed Proximities Between Hilbert Lobes}
	\label{fig:relaxed}
\end{figure}

\section{Application: Relaxed Descriptive Proximities}\label{application}
This application introduces a relaxed descriptive proximity (denoted by $\donear$), which is a refinement of the indefinite proximity relation 
%{\color{red}  
$\near_{\lim\Phi}$
%} 
in the proof of Lemma.~\ref{lemma:indefiniteDescriptively}. $\donear$ serves as means of pinpointing non-equal descriptively close segments of motion waveforms emanating from vibrating dynamical systems.  Let $A,B$ be measurable bounded regions in a typical motion waveform of a dynamical system. A practical outcome $\donear(A,B)$ is the detection those portions of system motion that are stable and have low energy dissipation.\\

A pair of nonempty sets $A,B$ have relaxed descriptive proximity (denoted by \boxed{A\ \donear\ B}, provided the characteristics of $A$ and $B$ are close enough.

\begin{definition}\label{def:donear}{\bf Relaxed Descriptive Proximity~\cite{ErdagPetersDeveci2024Dnear0}.}\\
Let $\varepsilon\in [0,1]$ and let $2^X$ denote the collection of subsets of a nonempty set $X$. The relaxed descriptive proximity relation $\donear$ is defined by
	\begin{center}
	\boxed{A\ \donear\ B\ \Leftrightarrow 
	\abs{\Phi(A)-\Phi(B)<\epsilon, A,B\subset X.}
 \mbox{\qquad \textcolor{blue}{$\blacksquare$}}}
	\end{center}
\end{definition}

The focus here is on proximal Hilbert envelope lobes inherent in time-constrained dynamical systems. Encapsulated in a Hilbert envelope, there are comparable regions called lobes.

\begin{definition}\label{def:HilbertEnvelope}{\bf (Hilbert Envelope Lobe).}\\  
A Hilbert envelope lobe (denoted by $H_{env}$) is a tiny bounded planar region attached to single waveform peak point on a waveform envelope, defined by
\begin{center}
\vspace{0.2cm}
\boxed{\boldsymbol{
H_{env} = \sqrt{m(t)^2 + (-m(t))^2}\mbox{\rm\ \cite{Brandt2011}}\mbox{
\qquad \textcolor{blue}{$\blacksquare$}}
}}
\end{center}
\vspace{0.2cm}
\end{definition}

\begin{remark}  
Hilbert envelope lobe regions provide a measure of motion waveform energy dissipation. The bigger a lobe region, the greater the motion energy.  Our interest here is in identifying those motion waveforms having a high number of proximal Hilbert lobe energy regions that have a corresponding low energy dissipation (denoted by $E_{diss}(t)$ at time $t$).  
\qquad \textcolor{blue}{$\blacksquare$}
\end{remark}
%(e.g., envelopes on a vibrating system waveforms).

Expenditure of energy $E_{m(t)}$ by a dynamical system $S_d$ is measured in terms of the area bounded by the motion $m(t)$ waveform emanating from $S_d$ at time $t$, i.e.,

\begin{definition}\label{axiom:waveformEnergy}{\bf (Waveform Energy).}\\
A measure of dynamical system energy is the area of a finite planar region bounded by system waveform $m(t)$ curve at time $t$,
defined by
%\vspace*{0.2cm}
\begin{center}
\boxed{\boldsymbol{
E_{m(t)} = \int_{t_0}^{t_1}\abs{m(t)}^2dt
}.}
\end{center}
\end{definition}

\begin{lemma}\label{lemma:energy}{\bf (Dynamical System Energy~\cite{PetersLiyanage2024}).}\\
Dynamical system energy is time-constrained and is always limited.
\end{lemma}

From Lemma~\ref{lemma:energy}, we obtain 

\begin{theorem}\label{theorem:Em}{\bf (Time-constrained Dynamical System Energy~\cite{PetersLiyanage2024} ).}\\
Let $X$ be a dynamical system with waveform $m(t)$ at time $t$. The energy of $X$ varies with every clock tick.
\end{theorem}

\begin{definition}\label{def:Ediss}{\bf (Energy Dissipation~\cite{PetersLiyanage2024}).}\\
For a motion waveform $m(t)$ at time $t$, let $loc$ identify a bounded region (called a Hilbert lobe) at location $x$ in an enveloped waveform and let $t$ be the time of occurence of the lobe. A measure of motion dissipated energy is the mapping\\ $E_{diss}:\mathbb{R}\times\mathbb{R}\to \mathbb{R}$, which is defined by
\[
E_{diss}(loc,t) = \abs{E(x,t)-E(x',t')}\geq 0\ \mbox{for quantized dissipation}.
\]
There are two different forms of motion energy dissipation to consider, namely,\\

\begin{compactenum}[1$^o$]
\item \textbf{Energy Dissipation within one frame}\\
Let $x,x'$ identify lobes in two different locations at time $t$ in an enveloped waveform in a single video frame at elapsed time $t$.  Then
\begin{center}
\vspace{0.1cm}
\boxed{E_{diss}(loc,t) = \abs{E(x,t)-E(x',t)}.
}
\end{center} 
\vspace{0.1cm}
\item {\bf Energy Dissipation between different frames}\\
Let $x=x'$ identify a lobe in the same location at times $t,t'$ in a pair video frames.  Then
\begin{center}
\boxed{E_{diss}(loc,t) = \abs{E(x,t)-E(x,t')}.
}\mbox{\qquad \textcolor{blue}{$\blacksquare$}
}
\end{center} 
\end{compactenum}
\end{definition}

\noindent The motivation for considering two different forms of $E_{diss}(loc,t)$ stems from Lemma~\ref{lemma:donearLobes}, which associates the relaxed descriptive proximity of a pair of energy regions with corresponding energy dissipation that is minimal.
 
%For a motion waveform $m(t)$ at time $t$, measure of motion dissipated energy $E_{diss}$ is defined in terms of the difference between energy $E(t)$ at time $t$ and energy $E(t')$at a later time t', i.e.,
%\\
%\textbf{Energy Dissipation within one frame}
%\begin{center}
%\boxed{E_{diss}(x-x',t) = \abs{E(x,t)-E(x',t)}.
%}
%\end{center} 
%\textbf{Energy Dissipation between different frame}
%\begin{center}
%\boxed{E_{diss}(x,t-t') = \abs{E(x,t)-E(x,t')}.
%}
%\end{center} 

%at time $t$ of a vibratory system.
%\vspace{0.2cm}
%The motivation for considering motion curves with a number descriptively proximal energy regions stems from the following conjecture.
%\vspace{0.2cm}

\begin{lemma}\label{lemma:donearLobes}
Let $E_{m(t)},E_{m(t'), t\neq t'}$ be the energy motion $m(t)$ at time $t$ and the energy of motion $m(t)$ of a dynamical system $S$. And let $E_{diss}(loc,t)$ be the energy dissipation $S$ at times $t,t'$.
\begin{center} 
$
E_{m(t)}\ \donear\ E_{m(t')}
\Leftrightarrow E_{diss}(loc,t)<\varepsilon\in [0,1].
$
\end{center}
That is, $E_{diss}(loc,t)$ is minimal.
\end{lemma}
\begin{proof}
It is known that energy dissipation $E_{diss}(t')$ in a motion curve is minimal, when the difference between $E_{m(t)},E_{m(t')}, t<t'$ is small~\cite{PetersLiyanage2024}. This occurs when $E_{m(t)}\ \donear\ E_{m(t')}$, i.e., $E_{diss}(t)<\varepsilon\in [0,1]$.  This gives the desired result.
\end{proof}

\begin{theorem}\label{theorem:lowEdiss}
Every Hilbert enveloped motion curves $m(t),m(t')$ with 
$E_{m(t)}\ \donear\ E_{m(t')}$ at times $t\neq t'$ has low energy dissipation.
\end{theorem}
\begin{proof}
Immediate from Lemma~\ref{lemma:donearLobes}.
\end{proof}

\begin{figure}[!ht]
	\centering
	\includegraphics[width=50mm]{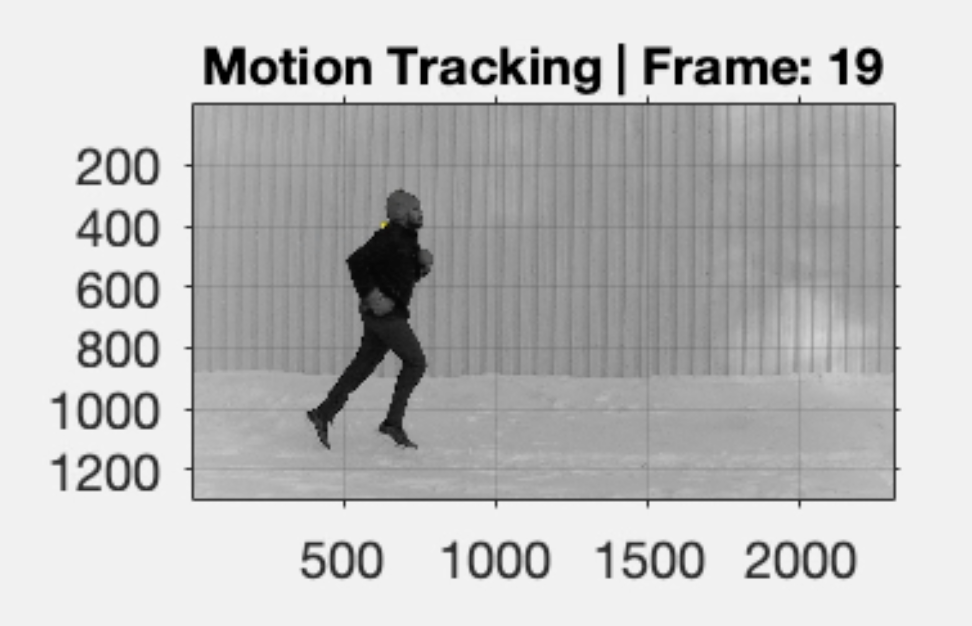}\quad\includegraphics[width=50mm]{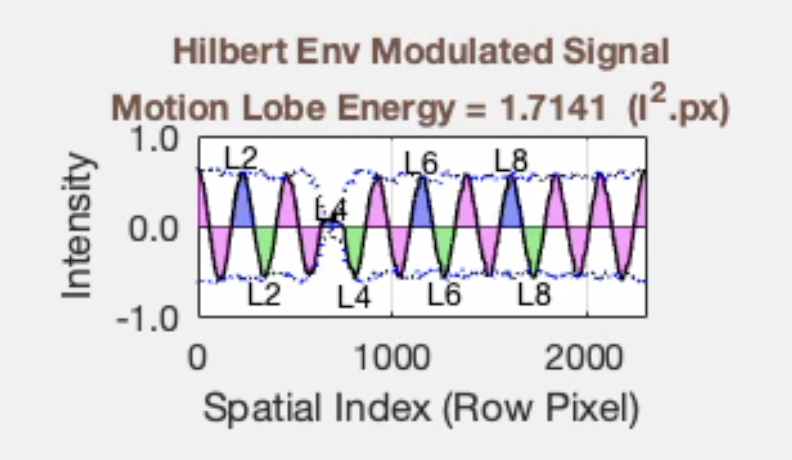}
	\caption{Source of Hilbert energy lobes in Table~\ref{table:runner19}}
	\label{fig:Ediss}
\end{figure} 

The sample lobe areas (aka motion energy regions) in Table~\ref{table:runner19} (see Figure~\ref{fig:Ediss}) lead to Example~\ref{ex:fr19}.

\begin{table}[htbp]%\label{table:runner19}
\centering
\caption{Lobe Energy for Runner in IR Video Frame 19}
\begin{tabular}{|c|c|c|c|}
\hline
\textbf{Lobe} & \textbf{+ve Lobe Areas} & \textbf{-ve Lobe Areas} & \textbf{Energy Dissipation} \\
           & \textbf{(+ve lobe)}                 & \textbf{(-ve lobe)}             & \boxed{E_{diss}}         \\
\hline
L2  & 38.952 & 36.2096 & 2.7424 \\
\hline
L4  & 2.522 & 32.3247 & 29.8027 \\
\hline
L6  & 33.1404 & 33.3126 & 0.1722 \\
\hline
L8 & 33.6902 & 33.9613 & 0.2711 \\
\hline
\end{tabular}
\label{table:runner19}
\end{table}

\begin{example}\label{ex:fr19}{\bf (Relaxed Proximities Within a Frame).}
Let $\epsilon = 0.03$.  From Table~\ref{table:runner19}, observe\\
\begin{compactenum}[1$^o$]
\item $L2^{+ve}\ \not{\donear}\ L2^{-ve}$.
\item $L4^{+ve}\ \not{\donear}\ L4^{-ve}$.
\item $L6^{+ve}\ \donear\ L6^{-ve}$.
\item $L8^{+ve}\ \donear\ L8^{-ve}$. \qquad \textcolor{blue}{$\blacksquare$}
\end{compactenum}
\end{example}
 	
\begin{figure}[!ht]
	\centering
	\includegraphics[width=50mm]{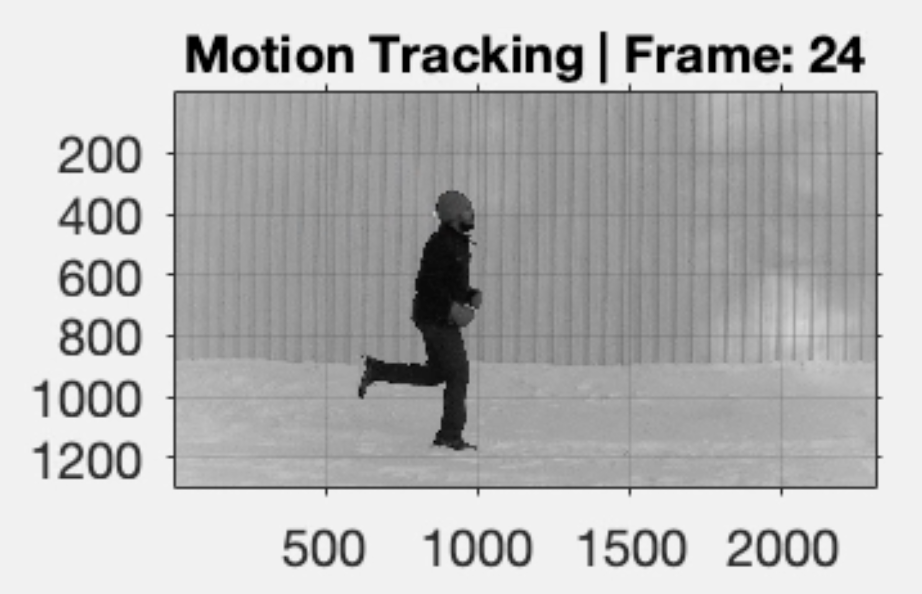}\quad\includegraphics[width=50mm]{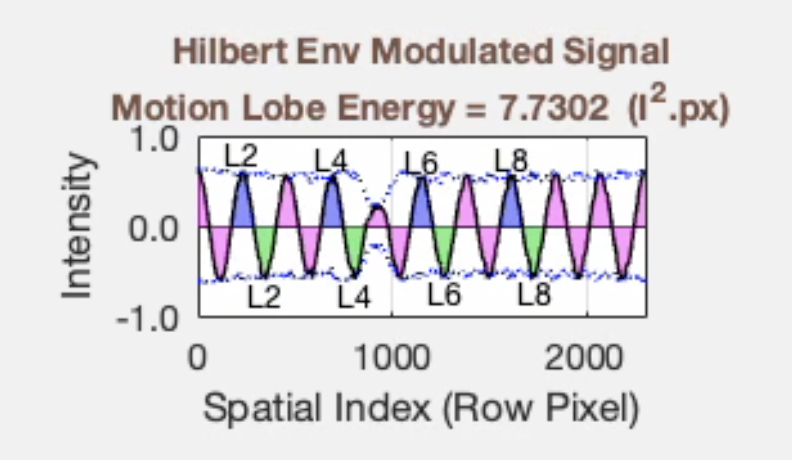}
	\caption{Source of Hilbert energy lobes in Table~\ref{table:runner24}}
	\label{fig:fr24Ediss}
\end{figure} 

The sample lobe areas (aka motion energy regions) in Table~\ref{table:runner24} (see Figure~\ref{fig:fr24Ediss}) and in Table~\ref{table:runner26} (see Figure~\ref{fig:fr26Ediss}) are a source of important relaxed proximities between Hilbert lobes in terms of lobe energy.  Lobe energy proximities between frames serve as an indication of the stability of system energy.  Relaxed proximities between IR frames are considered in Example~\ref{ex:fr2426}.

A check on relaxed proximities between Hilbert lobes in terms of lobe energy is considered next.

The sample lobe areas (aka motion energy regions) in Table~\ref{table:runner24} (see Figure~\ref{fig:fr24Ediss}).

\begin{table}[htbp]%\label{table:runner24}
\centering
\caption{Lobe Energy for Runner in IR Video Frame 24}
\begin{tabular}{|c|c|c|c|}
\hline
\textbf{Lobe} & \textbf{+ve Lobe Areas} & \textbf{-ve Lobe Areas} & \textbf{Energy Dissipation} \\
           & \textbf{(+ve lobe)}                 & \textbf{(-ve lobe)}             & \boxed{E_{diss}}         \\
\hline
L2  & 38.7384 & 36.3319 & 2.4065 \\
\hline
L4  & 34.5799 & 35.0868 & 0.5069 \\
\hline
L6  & 32.8082 & 33.3491 & 0.5409 \\
\hline
L8 & 33.0718 & 33.5572 & 0.4854 \\
\hline
\end{tabular}
\label{table:runner24}
\end{table} 

\begin{figure}[!ht]
	\centering
	\includegraphics[width=50mm]{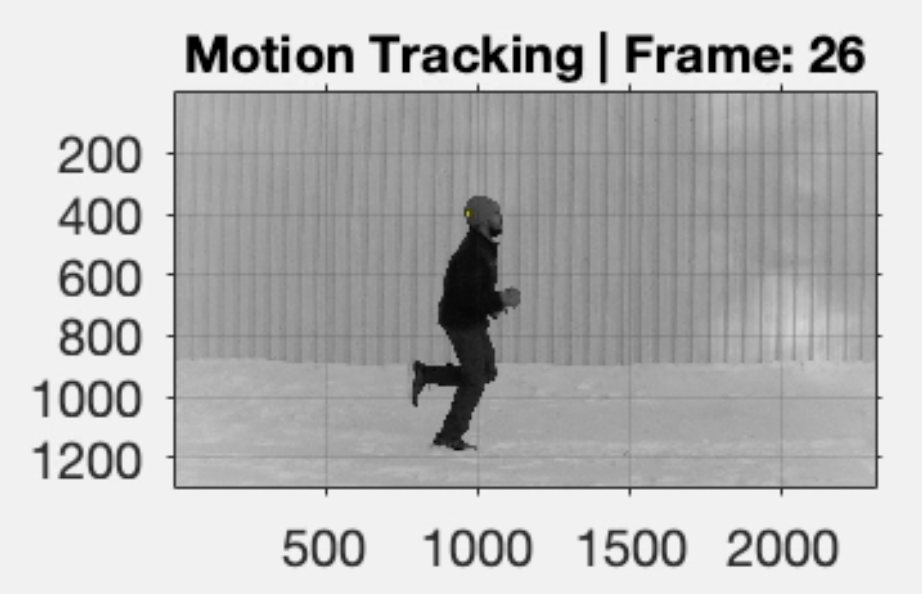}\quad\includegraphics[width=50mm]{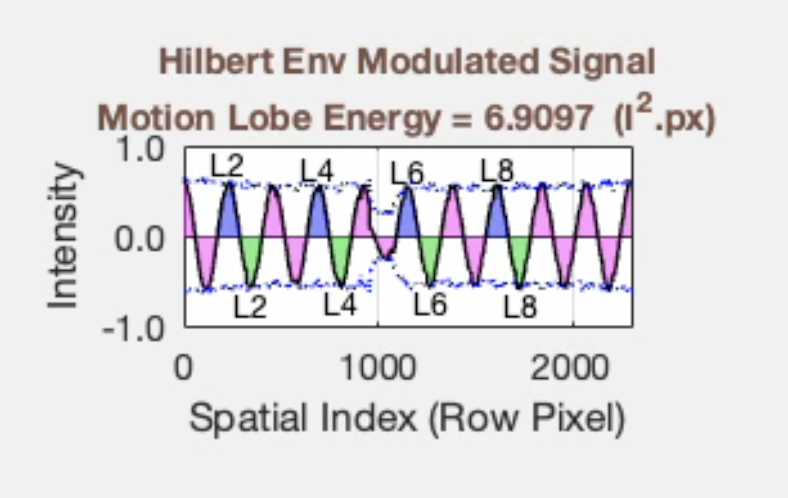}
	\caption{Source of Hilbert energy lobes in Table~\ref{table:runner26}}
	\label{fig:fr26Ediss}
\end{figure} 
  
\begin{table}[htbp]%\label{table:runner26}
\centering
\caption{Lobe Energy for Runner in IR Video Frame 26}
\begin{tabular}{|c|c|c|c|}
\hline
\textbf{Lobe} & \textbf{+ve Lobe Areas} & \textbf{-ve Lobe Areas} & \textbf{Energy Dissipation} \\
           & \textbf{(+ve lobe)}                 & \textbf{(-ve lobe)}             & \boxed{E_{diss}}         \\
\hline
L2  & 39.0058 & 36.4091 & 2.5967 \\
\hline
L4  & 34.3714 & 35.023 & 0.6516 \\
\hline
L6  & 32.8635 & 33.424 & 0.5605 \\
\hline
L8 & 33.1931 & 34.1811 & 0.988 \\
\hline
\end{tabular}
\label{table:runner26}
\end{table}	
 	
\begin{example}\label{ex:fr2426}{\bf (Relaxed Proximities Between Frames)}
Let $\epsilon = 0.2$ for the dynamical system represented by motion $m(t)$ in Figure~\ref{fig:fr24Ediss} and motion $m(t')$ in Figure~\ref{fig:fr26Ediss}.  Also, let\
\begin{center} 
\boxed{\mbox{(lobe) }L^{Tablek}_{k}, \mbox{(energy) }E^{Tablek}_{\pm L_k}, k\in\left\{19,24,26\right\}, n\in\left\{2,4,6,8\right\}
}
\end{center}  
Observe\\
\begin{compactenum}[1$^o$]
\item $+veL^{Table24}_{4}\ \donear\ +veL^{Table26}_{4}$, since\\
$\abs{E^{Table24}_{+L_4}-E^{Table26}_{+L_4}} =
\abs{34.5799-34.3714} = 0.2085 < \epsilon$.\\ 
\item $-veL^{Table24}_{4}\ \donear\ -veL^{Table26}_{4}$, since\\
$\abs{E^{Table24}_{+L_4}-E^{Table26}_{+L_4}} =
\abs{35.0868-35.023} = 0.0638 < \epsilon$.\\ 
\item $\abs{E^{Table24}_{+L_6}-E^{Table26}_{+L_6}} =
\abs{32.8082-32.8635} = 0.0553 < \epsilon$.\\ 
\item $-veL^{Table24}_{4}\ \donear\ -veL^{Table26}_{4}$, since\\
$\abs{E^{Table24}_{-L_6}-E^{Table26}_{-L_6}} =
\abs{33.3491-33.424} = 0.0749 < \epsilon$. \\
\end{compactenum}
Hence, $E_{m(t)}\ \donear\ E_{m(t')}$.
From Theorem~\ref{theorem:lowEdiss}, system $S$ at times $t$ and $t'$ has low energy dissipation. \qquad \textcolor{blue}{$\blacksquare$}
\end{example}

\end{document}